\documentclass[reqno,11pt]{amsart}
\usepackage[margin=1in]{geometry}

\usepackage{tikz}
\usetikzlibrary{shapes}
\usetikzlibrary{calc}
\usepackage{pgfplots}
\usepgfplotslibrary{fillbetween}

\usepackage{amsthm, amsmath, amssymb,xcolor,bm}

\usepackage[textsize=small,backgroundcolor=orange!20]{todonotes}

\usepackage[hidelinks]{hyperref}
\usepackage[shortlabels]{enumitem}

\usepackage[noabbrev,capitalize]{cleveref}
\crefname{equation}{}{}
\usepackage[color,final]{showkeys} %add in 'final' into parameter to remove showkeys

\colorlet{refkey}{orange!20}
\colorlet{labelkey}{blue!60}

\numberwithin{equation}{section}

\newtheorem{theorem}{Theorem}[section]
\newtheorem{proposition}[theorem]{Proposition}
\newtheorem{lemma}[theorem]{Lemma}

\newtheorem{corollary}[theorem]{Corollary}

\theoremstyle{definition}
\newtheorem{definition}[theorem]{Definition}

\newtheorem{open}[theorem]{Open Problem}
\newtheorem{question}[theorem]{Question}
\newtheorem{example}[theorem]{Example}

\theoremstyle{remark}
\newtheorem{remark}[theorem]{Remark}

\newcommand{\abs}[1]{\left\lvert#1\right\rvert}
\newcommand{\abss}[1]{\lvert#1\rvert}

\newcommand{\floor}[1]{\left\lfloor #1 \right\rfloor}

\newcommand{\paren}[1]{\left( #1 \right)}

\newcommand{\ol}{\overline}

\newcommand{\bA}{{\bm{A}}}
\newcommand{\balpha}{{\bm{\alpha}}}

\newcommand{\FF}{\mathbb{F}}
\newcommand{\RR}{\mathbb{R}}

\title{Which graphs can be counted in $C_4$-free graphs?}

\author{David Conlon}
\address{Conlon, Department of Mathematics, California Institute of Technology, Pasadena, CA, USA}
\email{dconlon@caltech.edu}
\thanks{Conlon is supported by NSF Award DMS-2054452.}

\author{Jacob Fox}
\address{Fox, Department of Mathematics, Stanford University, Stanford, CA, USA}
\email{jacobfox@stanford.edu}
\thanks{Fox is supported by a Packard Fellowship and by NSF Award DMS-1855635.}

\author{Benny Sudakov}
\address{Sudakov, Department of Mathematics, ETH, Z\"urich, 8092, Switzerland}
\email{benjamin.sudakov@math.ethz.ch}
\thanks{Sudakov is supported in part by SNSF grant 200021\_196965.}

\author{Yufei Zhao}
\address{Zhao, Department of Mathematics, Massachusetts Institute of Technology, Cambridge, MA, USA}
\email{yufeiz@mit.edu}
\thanks{Zhao is supported by NSF Award DMS-1764176, the MIT Solomon Buchsbaum Fund, and a Sloan Research
Fellowship.}

\begin{document}

\begin{abstract}
For which graphs $F$ is there a sparse $F$-counting lemma in $C_4$-free graphs?
We are interested in identifying graphs $F$ with the property that, roughly speaking, if $G$ is an $n$-vertex $C_4$-free graph with on the order of $n^{3/2}$ edges, then the density of $F$ in $G$, after a suitable normalization, is approximately at least the density of $F$ in an $\epsilon$-regular approximation of $G$. In recent work, motivated by applications in extremal and additive combinatorics, we showed that $C_5$ has this property.
Here we construct a family of graphs with the property.
\end{abstract}

\maketitle

\section{Introduction} \label{sec:intro}

When applying the regularity method in extremal graph theory, proofs can often be divided into two steps: first applying Szemer\'edi's regularity lemma to partition
a large graph so that most pairs of parts are regular and then using a counting (or embedding) lemma to find copies of a particular subgraph in this regular partition. For dense graphs, these steps are generally well-behaved and essentially completely understood. For sparse graphs, however, both steps can break down without additional hypotheses. Here we will focus on the second step of finding appropriate counting lemmas in the sparse regime, since the regularity step is now reasonably well understood~\cite{Koh97,Sco11} (although difficulties in maintaining the so-called no-dense-spots condition can arise even here). 

Similar issues arise in the study of \emph{quasirandom graphs}, a fundamental theme developed and popularized by Chung, Graham and Wilson~\cite{CGW89}, building on earlier work of Thomason~\cite{Tho87}. In their work, they showed, somewhat surprisingly, that several distinct notions of quasirandomness in dense graphs are essentially equivalent. In particular, in an $n$-vertex graph $G$ with edge density $p$, where $p$ is a fixed constant, having $C_4$-density $p^4 + o(1)$ is equivalent to a certain discrepancy condition and this in turn implies that the $F$-density in $G$ is $p^{\abs{E(F)}} + o(1)$ for all fixed graphs $F$. However, as already observed by Chung and Graham in~\cite{CG02}, these equivalences do not automatically carry over to graphs with $o(n^2)$ edges without additional assumptions. Indeed, 
even rather modest variants of the Chung--Graham--Wilson equivalences 
can fail to hold~\cite{SSTZ}. One viewpoint on our work here is that some aspect of the Chung--Graham--Wilson equivalences may be recovered if we assume that our graph is $C_4$-free.

Previous work on developing counting lemmas for sparse graphs has largely focused on controlling relatively dense subgraphs of sparse random or pseudorandom graphs. For instance, a counting lemma in sparse random graphs was proved by Conlon, Gowers, Samotij, and Schacht~\cite{CGSS14} in connection with the celebrated K{\L}R conjecture~\cite{KLR97} (see also~\cite{BMS15, ST15}), 
while a counting lemma in sparse pseudorandom graphs was proved by Conlon, Fox, and Zhao~\cite{CFZ14} 
and later extended to hypergraphs~\cite{CFZ15}, allowing them to simplify the proof of the Green--Tao theorem~\cite{GT08} (see also~\cite{CFZ14a} for a detailed exposition incorporating many further simplifications of the original proof).

In recent work~\cite{CFSZ}, motivated by applications in extremal and additive combinatorics, we pursued the study of sparse regularity in a very different setting, without any explicit pseudorandomness hypothesis. Instead, the only hypothesis on the host graph was that it be $C_4$-free. Under this assumption, we proved a $C_5$-counting lemma, which, when combined with an appropriate sparse regularity lemma, led to various new results, including a $C_5$-removal lemma in $C_4$-free graphs. As an example of an additive combinatorics application, we showed that every Sidon subset of $[N]$ without nontrivial solutions to $w + x + y + z = 4u$ has at most $o(\sqrt{N})$ elements. 
Here a \emph{Sidon set} is a set without nontrivial solutions to the equation $x + y = z + w$ and it is known that the maximum size of a Sidon subset of $[N]$ is $(1+o(1))\sqrt{N}$. We refer the interested reader to \cite{CFSZ} for further discussion of applications.

In this article, we continue the study of counting lemmas in $C_4$-free graphs, our main interest being the problem of determining which graphs $F$, besides $C_5$, satisfy an $F$-counting lemma in $C_4$-free graphs. We will make this question more precise in \cref{def:CL} below, when we say formally what it means for a graph $F$ to be \emph{countable}.

\begin{question}[Main question, informal] \label{que:main}
For which graphs $F$ is there an $F$-counting lemma in $C_4$-free graphs?	
\end{question}

By extending the proof of \cite[Theorem 1.1, see Section 4]{CFSZ} (which was written for $F = C_5$, but easily extends), we can deduce an $F$-removal lemma in $C_4$-free graphs whenever $F$ is countable. 

\begin{corollary}[Sparse removal lemma in $C_4$-free graphs] \label{cor:rem}
For any countable graph $F$ and any $\epsilon > 0$, there exists $\delta = \delta(F, \epsilon) > 0$ such that
every $n$-vertex $C_4$-free graph with at most $\delta n^{\abs{V(F)} - \abs{E(F)}/2}$ copies of $F$ can be made $F$-homomorphism-free by removing at most $\epsilon n^{3/2}$ edges.
\end{corollary}

Here ``copies of $F$'' refer to subgraphs isomorphic to $F$, whereas ``$F$-homomorphism-free'' means that there is no graph homomorphism from $F$ into the resulting graph after edge removal. 
In particular, if $F$ is bipartite and the number of copies of $F$ in a $C_4$-free graph on $n$ vertices is $o(n^{\abs{V(F)} - \abs{E(F)}/2})$, then $G$ has $o(n^{3/2})$ edges. 

Let us sketch the main ideas of the proof of \cref{cor:rem}, referring the reader to \cite[Section 4]{CFSZ} for further details.
We first apply a sparse weak regularity lemma to approximate the $C_4$-free graph $G$ by some ``dense'' graph $H$ (allowing edge-weights in $[0,1]$ for $H$).
The counting lemma then implies that $H$ has small $F$-homomorphism density.
By the dense $F$-removal lemma, applied as a black box, one can therefore remove a collection of edges from $H$ with small total weight so that the remaining graph contains no subgraphs to which $F$ is homomorphic.
Removing the corresponding edges from $G$ then makes it $F$-homomorphism-free.

The notion of having an $F$-counting lemma is made precise in the following definition. 
Note that the conclusion we seek is one-sided, that is, we only ask for a lower bound. In practice, this is usually all that is needed in applications. 

\begin{definition} \label{def:CL}
	A graph $F$ is \emph{countable} if, for every $\epsilon > 0$, there exists  $\delta = \delta(F, \epsilon) > 0$ such that if $G$ is an $n$-vertex $C_4$-free graph on vertex set $V$ and $H \in [0,1]^{V \times V}$ is a symmetric matrix (i.e., an edge-weighted graph) satisfying
	\begin{equation}\label{eq:disc}
	\abs{\frac{e_G(A,B)}{n^{3/2}} - \frac{e_H(A,B)}{n^2}} \le \delta \qquad \text{for all } A,B\subseteq V,
	\end{equation}
	(here $e_G(A,B) = \{(x,y) \in A\times B: xy \in E(G)\}$ and $e_H(A,B) = \sum_{x \in A, y \in B} H(x,y)$),
	then, for every $\bA = (A_v)_{v \in V(F)}$ with $A_v \subseteq V$ for each $v \in V(F)$, one has
	\begin{equation} \label{eq:lower-count}	
	\frac{\hom_\bA(F, G)}{n^{\abs{V(F)} - \abs{E(F)}/2}}
	\ge	\frac{\hom_\bA(F, H)}{n^{\abs{V(F)}}} - \epsilon,
	\end{equation}
	where $\hom_\bA(F,G)$ is the number of homomorphisms from $F$ to $G$ where each $v \in V(F)$ is mapped to a vertex in $A_v$ and $\hom_\bA(F,H)$ is the weighted analogue defined by the formula
	\[
	\hom_\bA(F,H) := \sum_{x_v \in A_v \ \forall v \in V(F)} \prod_{uv \in E(F)} H(x_u,x_v).
	\]
\end{definition}

The scaling in the denominators of the definition above is natural because the maximum number of edges in an $n$-vertex $C_4$-free graph is $(1/2 + o(1))n^{3/2}$ (see \cref{rem:polarity} below). It may be instructive to consider what happens when $G$ is the random graph $G(n,n^{-1/2})$ and $H$ is the all-$1$ matrix, in which case, provided $|E(F')| < 2|V(F')|$ for all subgraphs $F'$ of $F$,  \cref{eq:disc,eq:lower-count} with $\delta,\epsilon \to 0$ hold with high probability as $n \to \infty$.

\begin{remark}
In \cref{def:CL}, it suffices to only consider unweighted graphs $H$, since we can always randomly sample a weighted graph to get an unweighted graph with similar density properties.
However, in applications, $H$ is usually the normalized edge-density matrix of some (weak) regular partition of $G$, so it is more intuitive to allow edge-weights for $H$.
\end{remark}

\begin{remark} \label{rem:polarity} The \emph{polarity graph}~\cite{Bro66, ER62,ERS66} is an $n$-vertex $C_4$-free graph $G$ with $(1/2 + o(1))n^{3/2}$ edges (which is essentially best possible  by the K\H{o}v\'ari--S\'{o}s--Tur\'{a}n theorem~\cite{KST54}). In addition, it has the property that every edge lies in exactly one triangle and it satisfies the discrepancy condition~\cref{eq:disc} with $\delta = O(n^{-1/4})$ and $H$ being the all-$1$ matrix.

More specifically, let $q$ be a prime power and let $G_0$ be the graph with $q^2 + q + 1$ vertices, each corresponding to a point of the projective plane over $\FF_q$, i.e., elements of $\FF_q^3 \setminus \{(0,0,0)\}$ where $(x,y,z)$ is identified with $(\lambda x, \lambda y, \lambda y)$ for every nonzero $\lambda \in \FF_q$, with an edge between $(x,y,z)$ and $(x',y',z')$ if and only if $xx' + yy'+ zz' = 0$. This graph has exactly $q+1$ loops. It is also $(q+1)$-regular and has the property that each pair of distinct vertices has exactly one common neighbor, 
which in particular implies that $G_0$ is $C_4$-free. The square of its adjacency matrix is thus $q I + J$ (with $J$ being the all-$1$ matrix) and, hence, all of its eigenvalues, besides the top eigenvalue $q+1$, are $\pm \sqrt{q}$. 
The discrepancy claim in the previous paragraph then follows from the expander mixing lemma (see, e.g., \cite{KS:06}). In practice, we will actually use the induced subgraph $G$ of this graph where we remove all vertices with loops. This inherits the discrepancy property from $G_0$, but has the additional property mentioned above that every edge is contained in a unique triangle (see \cite{LV03} for a more detailed discussion of this point). 
\end{remark}

We now use the polarity graph to deduce a simple necessary condition for $F$ to be countable. 

\begin{remark} \label{rem:nec}
If $F$ is countable, then it has girth at least $5$. 

Indeed, suppose that $F$ contains a $4$-cycle $v_1 v_2 v_3 v_4$.
Let $G$ be an $n$-vertex polarity graph and $H$ the all-$1$ matrix.
The discrepancy property \cref{eq:disc} is satisfied for $\delta = o(1)$ by the previous remark. 
Set $A_{v_1}$, $A_{v_2}$, $A_{v_3}$, $A_{v_4}$ to be disjoint vertex sets of $V(G)$, each of order $\floor{n/4}$,
and $A_v = V(G)$ for all $v \in V(F)\setminus\{v_1, v_2,v_3,v_4\}$. 
Then $\hom_\bA(F, G) = 0$ since $G$ is $C_4$-free, 
but $\hom_\bA(F,H)\gtrsim n^{\abs{V(F)}}$,  
so $F$ is not countable.

Now suppose that $F$ contains a triangle.
Consider the graph $G'$ obtained from the polarity graph $G$ by deleting one edge from each triangle of $G$ chosen uniformly and independently at random (recall that $G$ is a disjoint union of triangles). With probability $1-o(1)$, the discrepancy property \cref{eq:disc} remains valid with $\delta = o(1)$ and $H$ the all-$2/3$ matrix.
However, \cref{eq:lower-count} fails when $A_v = V(G)$ for all $v$, since the fact that $G'$ is triangle-free implies that $\hom(F, G') = 0$. So again $F$ is not countable. (The same construction also appears in \cite[Lemma 2.6]{ABKS03}.)
\end{remark}

In the next section, we describe our main result, which gives a sufficient condition for countability, presented as a recursive construction.

\section{Countable graphs}

We begin with a simple proposition, whose proof may be found in \cref{sec:proof}. 

\begin{proposition} \label{prop:pendant-edge-countable}
	Adding a pendant edge to a countable graph produces a countable graph.
\end{proposition}

In particular, we have the following important corollary.

\begin{corollary}
	All trees are countable.
\end{corollary}

It will be shown in the next section that it suffices to verify countability within $n$-vertex $C_4$-free graphs $G$ with maximum degree at most $2\sqrt{n}$. This makes the following definition relevant.
 
\begin{definition} \label{def:tame}
	A graph $F$ is \emph{tame} if there exists a constant $C = C(F)$ such that $\hom(F, G) \le Cn^{\abs{V(F)}-\abs{E(F)}/2}$ for every $n$-vertex $C_4$-free graph $G$ with maximum degree at most $2\sqrt{n}$.
\end{definition}

An edgeless graph is clearly tame. 
Here is a sufficient recursive condition for tameness.

\begin{proposition} \label{prop:tame}
    Let $F$ be a tame graph. Let $F'$ be obtained from $F$ by either
	\begin{enumerate}[(a)]
	\item adding a pendant edge to $F$ (creating a single new leaf vertex) or
	\item joining two (not necessarily distinct) vertices of $F$ by a $3$-edge path whose two intermediate vertices are new. (If the two vertices of $F$ are the same, then the path is a triangle.)
	\end{enumerate}
	Then $F'$ is tame.
\end{proposition}

\begin{proof}
	Let $G$ be an $n$-vertex $C_4$-free graph with maximum degree at most $2\sqrt{n}$. It suffices to show that $\hom(F', G) \le 4 \sqrt{n} \hom(F, G)$. In case (a),  this is clear, since $G$ has maximum degree at most $2\sqrt{n}$. In case (b), we verify that the number of $3$-edge walks  between any pair of vertices (not necessarily distinct) in $G$ is at most $4\sqrt{n}$. Indeed, given $x,y \in V(G)$, let $w$ be a neighbor of $x$. If $w \ne y$ (at most $2\sqrt{n}$ such $w$), then, since $G$ is $C_4$-free, there is at most one 2-edge walk from $w$ to $y$. On the other hand, if $w = y$ (at most one such $w$), the number of $2$-edge walks from $w=y$ back to itself is $\deg(y) \le 2\sqrt{n}$.
\end{proof}

\begin{example} All cycles are tame, since, for each $\ell \ge 3$, one can first build an $(\ell-3)$-edge path using (a) and then complete it to an $\ell$-cycle using (b).
\end{example}

\tikzstyle{v}=[circle, fill=black, inner sep = 1pt]
\tikzstyle{p}=[fill=red, regular polygon, regular polygon sides=3, inner sep = 1.5pt]
\tikzstyle{i}=[draw=blue!50!black,fill=blue!25,dashed]

\begin{example}\label{ex:tame-sequence} The graphs in the sequence depicted below are also tame. To see this, observe that, at each step, we add a new path with $\ell \ge 3$ edges whose intermediate vertices are new (by again applying step (a) $\ell-3$ times and then applying step (b) once).

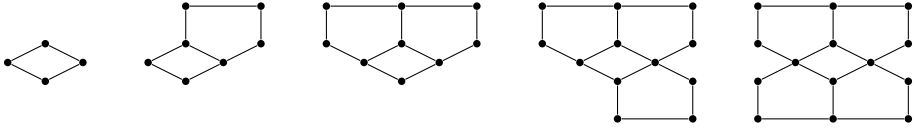
\[
\begin{tikzpicture}[scale=.25,baseline]
		\node[v] (r)   at ( 2, 0) {};
		\node[v] (u)   at ( 0, 1) {};
		\node[v] (l)   at (-2, 0) {};
		\node[v] (d)   at ( 0,-1) {};
		\draw (r)--(u)--(l)--(d)--(r);
\end{tikzpicture}
\qquad 
\begin{tikzpicture}[scale=.25,baseline]
		\node[v] (r)   at ( 2, 0) {};
		\node[v] (u)   at ( 0, 1) {};
		\node[v] (l)   at (-2, 0) {};
		\node[v] (d)   at ( 0,-1) {};
		\node[v] (uu)  at ( 0, 3) {};
		\node[v] (ruu) at ( 4, 3) {};
		\node[v] (ru)  at ( 4, 1) {};
		\draw (r)--(u)--(l)--(d)--(r);
		\draw (r)--(ru)--(ruu)--(uu)--(u);
\end{tikzpicture}
\qquad
\begin{tikzpicture}[scale=.25,baseline]
		\node[v] (r)   at ( 2, 0) {};
		\node[v] (u)   at ( 0, 1) {};
		\node[v] (l)   at (-2, 0) {};
		\node[v] (d)   at ( 0,-1) {};
		\node[v] (uu)  at ( 0, 3) {};
		\node[v] (ruu) at ( 4, 3) {};
		\node[v] (ru)  at ( 4, 1) {};
		\node[v] (luu) at (-4, 3) {};
		\node[v] (lu)  at (-4, 1) {};
		\draw (r)--(u)--(l)--(d)--(r);
		\draw (r)--(ru)--(ruu)--(uu)--(u);
		\draw (l)--(lu)--(luu)--(uu);
\end{tikzpicture}
\qquad
\begin{tikzpicture}[scale=.25,baseline]
		\node[v] (r)   at ( 2, 0) {};
		\node[v] (u)   at ( 0, 1) {};
		\node[v] (l)   at (-2, 0) {};
		\node[v] (d)   at ( 0,-1) {};
		\node[v] (uu)  at ( 0, 3) {};
		\node[v] (ruu) at ( 4, 3) {};
		\node[v] (ru)  at ( 4, 1) {};
		\node[v] (luu) at (-4, 3) {};
		\node[v] (lu)  at (-4, 1) {};
		\node[v] (dd)  at ( 0,-3) {};
		\node[v] (rdd) at ( 4,-3) {};
		\node[v] (rd)  at ( 4,-1) {};
		\draw (r)--(u)--(l)--(d)--(r);
		\draw (r)--(ru)--(ruu)--(uu)--(u);
		\draw (l)--(lu)--(luu)--(uu);
		\draw (r)--(rd)--(rdd)--(dd)--(d);
\end{tikzpicture}
\qquad
\begin{tikzpicture}[scale=.25,baseline]
		\node[v] (r)   at ( 2, 0) {};
		\node[v] (u)   at ( 0, 1) {};
		\node[v] (l)   at (-2, 0) {};
		\node[v] (d)   at ( 0,-1) {};
		\node[v] (uu)  at ( 0, 3) {};
		\node[v] (ruu) at ( 4, 3) {};
		\node[v] (ru)  at ( 4, 1) {};
		\node[v] (luu) at (-4, 3) {};
		\node[v] (lu)  at (-4, 1) {};
		\node[v] (dd)  at ( 0,-3) {};
		\node[v] (rdd) at ( 4,-3) {};
		\node[v] (rd)  at ( 4,-1) {};
		\node[v] (ldd) at (-4,-3) {};
		\node[v] (ld)  at (-4,-1) {};
		\draw (r)--(u)--(l)--(d)--(r);
		\draw (r)--(ru)--(ruu)--(uu)--(u);
		\draw (l)--(lu)--(luu)--(uu);
		\draw (r)--(rd)--(rdd)--(dd)--(d);
		\draw (l)--(ld)--(ldd)--(dd);
\end{tikzpicture}
\]
\end{example}

\begin{example} \label{ex:K23-not-tame}
	$K_{2,3}$ is not tame. Indeed, the $n$-vertex polarity graph $G$ has $\hom(K_{2,3}, G) \ge \hom(K_{1,3}, G) = \sum_{x \in V(G)} \deg_G(x)^3 \gtrsim n^{5/2}$, which is much larger than the $Cn^2$ upper bound required for tameness.
\end{example}

\begin{example} \label{ex:subdiv}
	Let $K'_k$ denote the $1$-subdivision of $K_k$. Then $K'_k$ is tame if and only if $k \le 4$. Indeed, let $G$ be the $n$-vertex polarity graph.
	Then, since there is a homomorphism $K'_k \to K_{1,\binom{k}{2}}$ mapping all $k$ vertices of the original $K_k$ to the same vertex,
	we have that 
	$$\hom(K'_k, G) \ge \hom(K_{1, \binom{k}{2}},G) \gtrsim n^{1 + \binom{k}{2}/2}.$$
	But $1 + \binom{k}{2}/2 > k=|V(K'_k)|-|E(K'_k)|/2$ for $k \ge 5$, so $K'_k$ is not tame.
	On the other hand, for $k \le 3$, $K'_k$ is tame due to \cref{prop:tame}, while, despite the fact that \cref{prop:tame} does not apply to $K'_4$, it is still tame, as may be verified by performing a case check based on which subsets of the original four vertices of $K_4$ are mapped to the same vertex.
	
	It will follow from our results below that every $K'_k$ is countable. Therefore, $K'_5$ (or $K'_k$ for any $k \ge 5$) is an example of a non-tame countable graph. 
	Moreover, since, for $H$ the all-$1$ matrix, the polarity graph $G$  satisfies the discrepancy property \cref{eq:disc} with $\delta = o(1)$, we see that $K'_5$ does not satisfy an ``upper-bound counting lemma'', i.e., \cref{eq:lower-count} with $\ge \dots -\epsilon$ replaced by $\le\dots + \epsilon$.
	That is, the $K'_5$-counting lemma in $C_4$-free graphs is truly one-sided.
\end{example}

We now describe an important building block in our recursive construction of countable graphs.

\begin{definition}
	Let $F$ be a graph and $I \subseteq V(F)$ an independent set.
	We say that $F$ is a \emph{connector with ends $I$} (or simply that $(F,I)$ is a connector) if
	\begin{enumerate}
		\item [(a)] $F$ is countable and
		\item [(b)] the graph $F \vee_I F$ formed by gluing two copies of $F$ along $I$ is tame.
	\end{enumerate} 
\end{definition}

Here is the simplest interesting connector.

\begin{example}
	The 2-edge path $v_0v_1v_2$ is a connector with ends $\{v_0, v_2\}$. This is illustrated below, where the ends of the connector are marked by red triangles.
	\[
	F = 
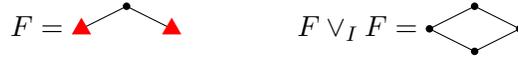
	\begin{tikzpicture}[scale=.3,baseline={(0,-.1)}]
		\draw (-2,0) node[p] {} -- (0,1) node[v] {} -- (2,0) node[p] {};
	\end{tikzpicture}
	\qquad \qquad 
	F \vee_I F = 	
	\begin{tikzpicture}[scale=.3,baseline={(0,-.1)}]
		\draw (-2,0) node[v] {} 
		-- (0,1) node[v] {} 
		-- (2,0) node[v] {} 
		-- (0,-1) node[v] {} 
		--cycle;
	\end{tikzpicture}
	\]
	More generally, any path is a connector with ends being any independent set. 	
	However, the same statement does not extend to all trees. For instance, $K_{1,3}$ does not give rise to a triple-ended connector,  since $K_{2,3}$ is not tame by \cref{ex:K23-not-tame}. 
\end{example}

Our main result is the following recursive construction of countable graphs. It can be visualized in terms of ``islands'' and ``bridges.'' We start with several disjoint tame countable components (the islands) and join them using connectors (the bridges). The theorem then says that the resulting graph is countable.

\begin{theorem} \label{thm:islands-bridges}
Let $F$ be a graph that is an edge-disjoint union of its subgraphs $F_1, \dots, F_k, J_1, \dots, J_\ell$, satisfying all of the following conditions:
\begin{enumerate}[(a)]
	\item $F_1, \dots, F_k$ are countable and vertex-disjoint;
	\item $F_1, \dots, F_{k-1}$ are tame ($F_k$ may be tame or not);
	\item for each $j \in [\ell]$, $J_j$ is a connector with ends $I_j = V(J_j) \cap V(F_1 \cup \cdots \cup F_k)$ and $I_j$ has at most one vertex in common with each $F_i$;
	\item each pair of connectors $J_i$ and $J_j$ share at most one vertex and the vertex they share (if any) lies in $I_i \cap I_j$.
\end{enumerate}
Then $F$ is countable.
\end{theorem}

\begin{example}
	The 5-cycle is countable. The ``islands and bridges'' decomposition is illustrated below, where each contiguous shaded region is an island.
	Both connectors are 2-edge-paths.
	
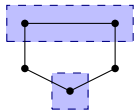
\[
	\begin{tikzpicture}[scale=.3,baseline]
		\fill[i] (-2.8,2.2) rectangle (2.8,3.8);
		\fill[i] (-0.8,-0.8) rectangle (0.8,0.8);
		\draw ( 0, 0) node[v] {} 
		   -- (-2, 1) node[v] {}
		   -- (-2, 3) node[v] {}
		   -- ( 2, 3) node[v] {}
		   -- ( 2, 1) node[v] {}
		   -- cycle;
	\end{tikzpicture}
	\]
	Similarly, $\ell$-cycles, for $\ell \ge 5$, can be shown to be countable by starting with two islands, one an isolated vertex, as above, and the other a path of length $\ell - 4$, with $2$-edge-path connectors joining the  endpoints of this path to the isolated vertex. As mentioned in~\cite[Footnotes 1 and 3]{CFSZ}, knowing that longer cycles can be counted allows us to extend our results~\cite[Section 1.3]{CFSZ} about finding solutions of  translation-invariant equations in Sidon sets to equations with more than five variables.
\end{example}

\begin{example}
	Since the 5-cycle is both countable and tame, we can use it as an island to build up further countable graphs.
	For example, connecting a pair of 5-cycles using 2-edge-path connectors, as shown below, yields a new countable graph.
	
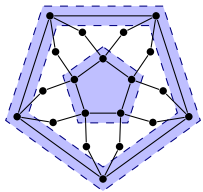
\[
 	 \begin{tikzpicture}[scale=.4,baseline]
 	 	\foreach \i in {0,1,2,3,4}{
	 		\coordinate (A\i) at (90+72*\i:1.4);
	 		\coordinate (D\i) at (54+72*\i:3.4) {};
	 		\coordinate (E\i) at (54+72*\i:2.6) {};
	 	}
	 	
		\fill[i,even odd rule] 
		(A0)--(A1)--(A2)--(A3)--(A4)--cycle
		(D0)--(D1)--(D2)--(D3)--(D4)--cycle
 	 	(E0)--(E1)--(E2)--(E3)--(E4)--cycle;

	 	\foreach \i in {0,1,2,3,4}{
	 		\node[v] (a\i) at (90+72*\i:1) {};
	 		\node[v] (b\i) at (70+72*\i:2) {};
	 		\node[v] (c\i) at (110+72*\i:2) {};
	 		\node[v] (d\i) at (54+72*\i:3) {};
	 	}
	 	\foreach \i[evaluate=\i as \j using {int(Mod(\i+1,5))}] in {0,1,2,3,4}{
	 		\draw (a\i) -- (a\j);
	 		\draw (d\i) -- (d\j);
	 		\draw (a\i) -- (b\i)--(d\i);
	 		\draw (a\i)--(c\i)--(d\j);
	 	}
	 \end{tikzpicture}
	\]
\end{example}
	
\begin{example}
	Using that the $5$-cycle is countable and tame, we see that the following graph is also countable, again with the islands shaded:
	
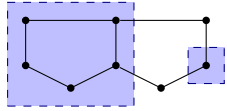
\[
	\begin{tikzpicture}[scale=.3,baseline]
		\fill[i] (-2.8,-0.8) rectangle (2.8,3.8);
		\fill[i] (5.2,0.2) rectangle (6.8,1.8);
		\draw ( 0, 0) node[v] {} 
		   -- (-2, 1) node[v] {}
		   -- (-2, 3) node[v] {}
		   -- ( 2, 3) node[v] (a1) {}
		   -- ( 2, 1) node[v] (b1) {}
		   -- cycle;
		\draw (a1) -- ++(4,0) node[v] (a2) {}
				   -- ++(0,-2) node[v] (b2) {}
				   -- ++(-2,-1) node[v] {}
				   -- (b1);
	\end{tikzpicture}
	\]
	This graph is also tame by \cref{prop:tame}, so we can repeat the  process to show that the following graph (and any longer chain of $5$-cycles) is tame and countable. 
	
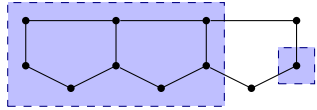
\[
	\begin{tikzpicture}[scale=.3,baseline]
	    \fill[i] (-2.8,-0.8) rectangle (6.8,3.8);
		\fill[i] (9.2,0.2) rectangle (10.8,1.8);
		\draw ( 0, 0) node[v] {} 
		   -- (-2, 1) node[v] {}
		   -- (-2, 3) node[v] {}
		   -- ( 2, 3) node[v] (a1) {}
		   -- ( 2, 1) node[v] (b1) {}
		   -- cycle;
		\draw (a1) -- ++(4,0) node[v] (a2) {}
				   -- ++(0,-2) node[v] (b2) {}
				   -- ++(-2,-1) node[v] {}
				   -- (b1);
   		\draw (a2) -- ++(4,0) node[v] {}
				   -- ++(0,-2) node[v] {}
				   -- ++(-2,-1) node[v] {}
				   -- (b2);
	\end{tikzpicture}
	\]
\end{example}

\begin{example}
	The following graph is a connector (with the ends again marked by red triangles):
	
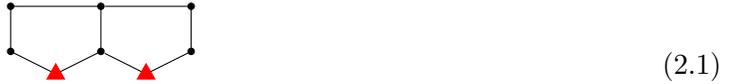
\begin{equation}\label{eq:c5pair-connector}
	\begin{tikzpicture}[scale=.3,baseline]
		\draw ( 0, 0) node[p] {} 
		   -- (-2, 1) node[v] {}
		   -- (-2, 3) node[v] {}
		   -- ( 2, 3) node[v] (a1) {}
		   -- ( 2, 1) node[v] (b1) {}
		   -- cycle;
		\draw (a1) -- ++(4,0) node[v] (a2) {}
				   -- ++(0,-2) node[v] (b2) {}
				   -- ++(-2,-1) node[p] {}
				   -- (b1);
	\end{tikzpicture}
	\end{equation}
	Indeed, we saw in the last example that this graph is countable, while the graph formed by gluing two copies along the ends, as shown below, is tame by \cref{ex:tame-sequence}.
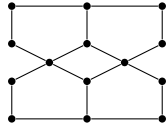
	\[
	\begin{tikzpicture}[scale=.25,baseline]
		\node[v] (r)   at ( 2, 0) {};
		\node[v] (u)   at ( 0, 1) {};
		\node[v] (l)   at (-2, 0) {};
		\node[v] (d)   at ( 0,-1) {};
		\node[v] (uu)  at ( 0, 3) {};
		\node[v] (ruu) at ( 4, 3) {};
		\node[v] (ru)  at ( 4, 1) {};
		\node[v] (luu) at (-4, 3) {};
		\node[v] (lu)  at (-4, 1) {};
		\node[v] (dd)  at ( 0,-3) {};
		\node[v] (rdd) at ( 4,-3) {};
		\node[v] (rd)  at ( 4,-1) {};
		\node[v] (ldd) at (-4,-3) {};
		\node[v] (ld)  at (-4,-1) {};
		\draw (r)--(u)--(l)--(d)--(r);
		\draw (r)--(ru)--(ruu)--(uu)--(u);
		\draw (l)--(lu)--(luu)--(uu);
		\draw (r)--(rd)--(rdd)--(dd)--(d);
		\draw (l)--(ld)--(ldd)--(dd);
	\end{tikzpicture}
	\]
	Similarly, we can check that the following graph (and any longer chain of $5$-cycles) is a multi-ended connector:
	
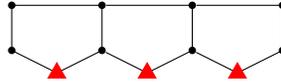
\[
	\begin{tikzpicture}[scale=.3,baseline]
		\draw ( 0, 0) node[p] {} 
		   -- (-2, 1) node[v] {}
		   -- (-2, 3) node[v] {}
		   -- ( 2, 3) node[v] (a1) {}
		   -- ( 2, 1) node[v] (b1) {}
		   -- cycle;
		\draw (a1) -- ++(4,0) node[v] (a2) {}
				   -- ++(0,-2) node[v] (b2) {}
				   -- ++(-2,-1) node[p] {}
				   -- (b1);
   		\draw (a2) -- ++(4,0) node[v] {}
				   -- ++(0,-2) node[v] {}
				   -- ++(-2,-1) node[p] {}
				   -- (b2);
	\end{tikzpicture}
	\]
\end{example}

\begin{example}
	The following graph is countable (one of the connectors is a 2-edge-path, while the other is \cref{eq:c5pair-connector}):
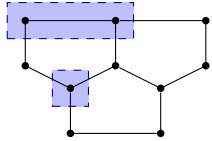
	\[
	\begin{tikzpicture}[scale=.3,baseline]
		\fill[i] (-2.8,2.2) rectangle (2.8,3.8);
		\fill[i] (-0.8,-0.8) rectangle (0.8,0.8);
		\draw ( 0, 0) node[v] (w1) {} 
		   -- (-2, 1) node[v] {}
		   -- (-2, 3) node[v] {}
		   -- ( 2, 3) node[v] (a1) {}
		   -- ( 2, 1) node[v] (b1) {}
		   -- cycle;
		\draw  (a1) -- ++(4,0) node[v]  {}
				    -- ++(0,-2) node[v]  {}
				    -- ++(-2,-1) node[v] (w2) {}
				    -- (b1);
		\draw  (w1) -- ++(0,-2) node[v] {}
					-- ++(4,0) node[v] {}
					-- (w2);
	\end{tikzpicture}
	\]
	We can extend this example further. Since the above graph is countable, we can use \cref{prop:tame} to verify that, with the ends as marked, it is also a connector:
	
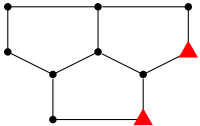
\[
	\begin{tikzpicture}[scale=.3,baseline]
		\draw ( 0, 0) node[v] (w1) {} 
		   -- (-2, 1) node[v] {}
		   -- (-2, 3) node[v] {}
		   -- ( 2, 3) node[v] (a1) {}
		   -- ( 2, 1) node[v] (b1) {}
		   -- cycle;
		\draw  (a1) -- ++(4,0) node[v]  {}
				    -- ++(0,-2) node[p]  {}
				    -- ++(-2,-1) node[v] (w2) {}
				    -- (b1);
		\draw  (w1) -- ++(0,-2) node[v] {}
					-- ++(4,0) node[p] {}
					-- (w2);
	\end{tikzpicture}
	\]
	Using this connector, we deduce that the following graph is countable:
	
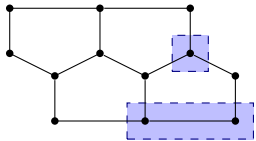
\[
	\begin{tikzpicture}[scale=.3,baseline]
		\fill[i] (3.2,-1.2) rectangle (8.8,-2.8);
		\fill[i] (5.2,0.2) rectangle (6.8,1.8);
		\draw ( 0, 0) node[v] (w1) {} 
		   -- (-2, 1) node[v] {}
		   -- (-2, 3) node[v] {}
		   -- ( 2, 3) node[v] (a1) {}
		   -- ( 2, 1) node[v] (b1) {}
		   -- cycle;
		\draw  (a1) -- ++(4,0) node[v]  {}
				    -- ++(0,-2) node[v]  (q){}
				    -- ++(-2,-1) node[v] (w2) {}
				    -- (b1);
		\draw  (w1) -- ++(0,-2) node[v] {}
					-- ++(4,0) node[v] (p) {}
					-- (w2);
		\draw (p) -- ++(4,0) node[v] {}
		          -- ++(0,2) node[v] {}
		          -- (q);
	\end{tikzpicture}
	\]
	Similar inductive arguments allow us to prove the countability of many other graphs of girth at least $5$. However, 
	as we shall explain in more detail in the concluding remarks, we are far from a classification. For instance, our methods seem insufficient for showing that $3$-regular graphs such as those below are countable.
\end{example}

\begin{open} \label{open:dod-pet}
	 Are the dodecahedral and Petersen graphs, shown below, countable? 
	 
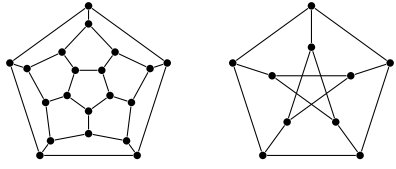
\[
	 \begin{tikzpicture}[scale=.2,baseline]
	 	\foreach \i in {0,1,2,3,4}{
	 		\node[v] (a\i) at (-90+72*\i:1.5) {};
	 		\node[v] (b\i) at (-90+72*\i:3) {};
	 		\node[v] (c\i) at (-54+72*\i:4.3) {};
	 		\node[v] (d\i) at (-54+72*\i:5.5) {};
	 	}
	 	\foreach \i[evaluate=\i as \j using {int(Mod(\i-1,5))}] in {0,1,2,3,4}{
	 		\draw (a\i) -- (a\j);
	 		\draw (a\i) -- (b\i);
	 		\draw (b\i) -- (c\i);
	 		\draw (b\i) -- (c\j);
	 		\draw (c\i) -- (d\i);
	 		\draw (d\i) -- (d\j);
	 	}
	 \end{tikzpicture}
	 \qquad 
 	 \begin{tikzpicture}[scale=.55,baseline]
	 	\foreach \i in {0,1,2,3,4}{
	 		\node[v] (a\i) at (90+72*\i:1) {};
	 		\node[v] (b\i) at (90+72*\i:2) {};
	 	}
	 	\foreach \i[evaluate=\i as \ia using {int(Mod(\i+2,5))},evaluate=\i as \ib using {int(Mod(\i+1,5))}] in {0,1,2,3,4}{
	 		\draw (a\i) -- (b\i);
	 		\draw (a\i) -- (a\ia);
	 		\draw (b\i) -- (b\ib);
	 	}
	 \end{tikzpicture}
	 \]
\end{open}

In the remainder of the article, we prove \cref{prop:pendant-edge-countable,thm:islands-bridges}.

\section{Trimming high-degree vertices}

In this brief section, we show that in the definition of countability, \cref{def:CL}, we can restrict to considering $n$-vertex $C_4$-free graphs $G$ satisfying an additional maximum degree assumption, namely, that $G$ has maximum degree at most $2\sqrt{n}$, without affecting the family of graphs which are countable.

\begin{lemma}
	Let $G$ be a graph on a vertex set $V$ of size $n$ and
	let $H \in [0,1]^{V \times V}$ be a symmetric matrix such that
	\begin{equation}
		\label{eq:trim-disc}
		\abs{\frac{e_G(A,B)}{n^{3/2}} - \frac{e_H(A,B)}{n^2}} \le \delta \qquad \text{for all } A,B \subseteq V.
	\end{equation}
	Let $S = \{v \in V : \deg_G(v) \le 2\sqrt{n}\}$ and
	let $G'$ be the subgraph of $G$ with the same vertex set $V$ but only keeping edges with both endpoints in $S$.
	Then
	\[
		\abs{\frac{e_{G'}(A,B)}{n^{3/2}} - \frac{e_H(A,B)}{n^2}} \le 3\delta \qquad \text{for all } A,B \subseteq V.
	\]
\end{lemma}

\begin{proof}
	Write $\ol S = V \setminus S$. Applying \cref{eq:trim-disc} to $(A,B) = (\ol S, V)$, we have
	\[
	\delta n^2 \geq \sqrt{n} e_G(\ol S, V) - e_H(\ol S, V) \geq \sqrt{n} \cdot 2\sqrt{n}\abss{\ol S} -\abss{\ol S}|V| = n\abss{\ol S},
	\]
	so $\abss{\ol S} \le \delta n$. 
	For any $A,B \subseteq V$, writing $A' = A \cap S$ and $B'= B \cap S$, 
	we have $e_{G'}(A,B) =  e_{G}(A',B')$, so
	\begin{align*}
	\abs{\sqrt{n} e_{G'}(A,B) - e_H(A,B)}
	&= \abs{\sqrt{n} e_{G}(A',B') - e_H(A',B') + e_H(A',B') - e_H(A, B)}
	\\
	&\le \abs{\sqrt{n} e_{G}(A',B') - e_H(A',B')} + (\abss{A \setminus A'} + \abss{B \setminus B'}) n 
	\\
	&\le \delta n^2 + 2\abss{\ol S} n
	\le 3\delta n^2. \qedhere
	\end{align*}
\end{proof}

\section{Notation and setup} \label{sec:notation}

Given a graph $F$, a \emph{vertex weight function} on $F$ (sometimes we say ``on $V(F)$'', as graphs and their vertex sets are interchangeable for this purpose) is a collection $\balpha = (\alpha_v)_{v \in V(F)}$ of functions $\alpha_v \colon V \to [0,1]$ indexed by $v$. It will be important for our arguments that each $\alpha_v$ takes values in $[0,1]$ and not in some wider range.

Let $x = (x_v)_{v \in V(F)} \in V^{V(F)}$ with $x_v \in V$.
For each $S \subseteq V(F)$, we write $x_S = (x_v)_{v \in S}$ for its projection onto the  coordinates indexed by $S$.
To avoid notational clutter, we will sometimes write a subgraph as the subscript rather than its vertex set. For example, if $F'$ is a subgraph of $F$ and $S \subseteq V(F)$, then we write $x_{F'} = x_{V(F')}$, $x_{F \setminus F'} = x_{V(F) \setminus V(F')}$, and $x_{F\setminus S} = x_{V(F)\setminus S}$.

Given a function $f \colon V^S \to \RR$, we write
\[
\int f(x_S) d x_S = \abs{V}^{-\abs{S}} \sum_{x_S \in V^S} f(x_S).
\]
Furthermore, given a vertex weight function $\balpha = (\alpha_v)_{v \in S}$ on $S$, we write
\[
\int f(x_S) d^\balpha x_S = \int f(x_S) \prod_{v \in S} \alpha_v(x_v) \, d x_S.
\]

Given a symmetric function $g \colon V \times V \to \RR$ and $x \in V^{V(F)}$, we define $g_F \colon V^{V(F)} \to \RR$ by
\[
g_F(x) = \prod_{uv \in F} g(x_u, x_v).
\]
Given $S \subseteq V(F)$ and a vertex weight function $\balpha$ on $F \setminus S$, we define $g_{F,S} \colon V^S \to \RR$ by
\[
g_{F,S}^\balpha(x_S) = \int g_F(x_F)\, d^\balpha x_{F \setminus S},
\]
which (up to normalization) corresponds to counting homomorphisms $F\to G$ where the image of $S$ is $x_S$ and the remaining vertices of $F$ are weighted by $\balpha$.
Such quantities also arise naturally when using flag algebras. Finally, given a vertex weight function $\balpha$ on $F$, we write
\[
t^\balpha(F, g) = g_{F, \emptyset}^\balpha = \int \prod_{uv \in F} g(x_u, x_v) \prod_{v \in V(F)} \paren{\alpha_v(x_v) dx_v},
\]
which is the $\balpha$-weighted homomorphism density of $F$ in $g$.

It will also be convenient to allow our weight function notation to be a little more flexible, in the sense that we automatically ignore uninvolved vertices. For example, if $\balpha$ is a vertex weight function on $F$ and $F'$ is a subgraph on a proper vertex subset, then we still write $t^{\balpha}(F', g)$ and $d^\balpha x_{F'}$ with the understanding that $\balpha$ is now restricted to the vertex set of $F'$. This way we do not always have to specify the set of vertices that the weight function is defined on.

\medskip

Both the discrepancy condition \cref{eq:disc} and the counting lemma conclusion \cref{eq:lower-count} can be equivalently rephrased in terms of weight functions $\balpha$ 
rather than product sets $\bA$. 
The extra flexibility allowed by considering $[0,1]$-valued weight functions will be helpful in our proofs.
To see the equivalence, note that, with the function $g = \sqrt{n} G$ (here we view $G \colon V \times V \to \{0,1\}$ as the edge-indicator function of the graph $G$), we have
\[
\frac{\hom_\bA(F, G)}{n^{\abs{V(F)} - \abs{E(F)}/2}} = t^\balpha (F, g)
\]
for the vertex weight function $\balpha$ on $F$ which is equal to the indicator function of $\bA$ (i.e., $\alpha_v(x) = 1$ if $x \in A_v$ and $0$ otherwise). Likewise, for $h = H$, 
\[
\frac{\hom_\bA(F, H)}{n^{\abs{V(F)}}} = t^\balpha(F, h).
\]
Hence, the counting lemma conclusion  \cref{eq:lower-count}, that
\[
	\frac{\hom_\bA(F, G)}{n^{\abs{V(F)} - \abs{E(F)}/2}}
	\ge	\frac{\hom_\bA(F, H)}{n^{\abs{V(F)}}} - \epsilon,
\]
is equivalent to the statement that
\begin{equation}
	\label{eq:t-counting}
	t^\balpha (F, g) \ge t^\balpha (F, h) - \epsilon
\end{equation}
for any $\{0,1\}$-valued vertex weight function $\balpha$.
Since $t^\balpha (F, g) -  t^\balpha (F, h)$ is a multilinear function of the values $(\alpha_v(x))_{v \in F, x \in V}$, 
the extrema of the function are attained when $\alpha_v(x) \in \{0,1\}$ for all $v \in F$ and $x \in V$.
This shows that the counting lemma conclusion \cref{eq:lower-count} is equivalent to the statement that \cref{eq:t-counting} holds for all vertex weight functions.

By the same argument, the discrepancy condition~\cref{eq:disc}, that 
\[
\abs{\frac{e_G(A,B)}{n^{3/2}} - \frac{e_H(A,B)}{n^2}} \le \delta \qquad \text{for all } A,B\subseteq V,
\]
is equivalent to
\[
\abs{\int (g-h)(x,y) \alpha_1(x)\alpha_2(y) dxdy} \le \delta \qquad \text{for all }\alpha_1,\alpha_2 \colon V \to [0,1].
\]
In fact, (thanks to the trimming step in the previous section) from now on we will only need the one-sided discrepancy hypothesis
\begin{equation}\label{eq:disc-lower}
\int g(x,y) \alpha_1(x)\alpha_2(y) dxdy
 \ge 
 \int h(x,y) \alpha_1(x)\alpha_2(y) dxdy - \delta \qquad \text{for all }\alpha_1,\alpha_2 \colon V \to [0,1].	
\end{equation}

\emph{Summary of what needs to be proved.}
To prove that $F$ is countable, it suffices to show that there is a constant $c > 0$ such that for every $\epsilon > 0$ there exists  $\delta > 0$ satisfying the following. 
Let $G$ be an $n$-vertex $C_4$-free graph on vertex set $V$ with maximum degree at most $2\sqrt{n}$. 
Let $g = c\sqrt{n} G$ and 
let $h \colon V \times V \to [0,1]$ be a symmetric function satisfying \cref{eq:disc-lower}. 
Then, for every vertex weight function $\balpha$ on $F$, one has \cref{eq:t-counting}.

The reason that we scale by a factor of $c$ in defining $g$ is so that the various tameness hypotheses on subgraphs of $G$ can be made to have the form $t(F', g) \le 1$. Furthermore, as long as $c \le 1/2$, the hypothesis that $G$ has maximum degree at most $2\sqrt{n}$ implies that 
\begin{equation} \label{eq:g-max-degree}
	\int g(x,y) \, dy \le 1 \qquad \text{ for all } x \in V.		
\end{equation}

\section{Counting lemma proofs} \label{sec:proof}

We follow without further comment the framework discussed in the previous section.

\begin{proof}[Proof of \cref{prop:pendant-edge-countable} (adding a pendant edge preserves countability)]
	Let $F$ be a graph with a leaf vertex $u$. Let $F'$ be $F$ with $u$ removed and
	assume that $F'$ is countable. Suppose that
	\begin{equation}
		\label{eq:leaf-disc}
	\int g(x,y) \alpha_1(x)\alpha_2(y) dxdy
	 \ge 
 \int h(x,y) \alpha_1(x)\alpha_2(y) dxdy - \epsilon
	\end{equation}
 for all $\alpha_1,\alpha_2 \colon V \to [0,1]$.	
	Since $F'$ is countable, we may also assume that
	\begin{equation} \label{eq:F'-hyp}
	t^{\balpha'}(F',g) \ge t^{\balpha'}(F',h) - \epsilon		
	\end{equation}
	for every vertex weight function $\balpha'$ on $F'$. 
	
	It suffices to show that these two inequalities imply that
	\begin{equation} \label{eq:pendant-edge-goal}		
	t^{\balpha}(F,g) \ge t^{\balpha}(F,h) - 2\epsilon
	\end{equation}
	for every vertex weight function $\balpha$ on $F$.
	For this, define a vertex weight function $\balpha'$ on $F'$ by $\alpha'_v = \alpha_v$ unless $v$ is the neighbor $v$ of $u$, in which case $\alpha'_{v}(x_v) = \alpha_{v}(x_v)\int g(x_v,x_u)\alpha_u(x_u)dx_u \in [0,1]$ by \cref{eq:g-max-degree}.
	Then,  
	by \cref{eq:F'-hyp} applied with this $\balpha'$,
	\[
	t^\balpha(F, g) = t^{\balpha'}(F',g) \ge t^{\balpha'}(F',h) - \epsilon.
	\]
	Furthermore, we have 
	\begin{align*}
	t^{\balpha'}(F',h) 
	&= \int h^\balpha_{F',v}(x_v) g(x_v,x_u) \alpha_u(x_u)\alpha_v(x_v) \, dx_u dx_v
	\\
	&
	\ge \int h^\balpha_{F',v}(x_v) h(x_v,x_u) \alpha_u(x_u)\alpha_v(x_v) \, dx_u dx_v
	 - \epsilon
	\\
	&
	= t^{\balpha}(F,h) - \epsilon,
	\end{align*}
	where the inequality step uses \cref{eq:leaf-disc}.
	Combining the last two displayed inequalities yields \cref{eq:pendant-edge-goal}, as desired.
\end{proof}

\begin{proof}[Proof of \cref{thm:islands-bridges} (islands and bridges)]
By the tameness assumptions, we can choose a sufficiently small constant $c \in (0,1]$ (depending only on $F$) such that, setting $g = c\sqrt{n} G \colon V \times V \to [0,\infty)$, we have 
\begin{equation}\label{eq:bounded}
t(F_i, g) \le 1 \text{ for all }  i \in [k-1] \quad \text{ and } \quad 
t(J_j \vee_{I_j} J_j, g) \le 1  \text{ for all }  j \in [\ell].
\end{equation}
Let $\epsilon \in (0,1]$ and let 
\begin{equation}
	\label{eq:eta}
\eta_i = \epsilon^{2^i} \text{ for each } i \in [\ell] 
\quad \text{ and } \quad 
\eta = \epsilon^{2^{\ell+1}}.
\end{equation}
By the countability assumption on $F_1, \dots, F_k, J_1, \dots, J_\ell$
it suffices to show that if $h \colon V \times V \to [0,1]$ satisfies
\begin{equation} \label{eq:recur-assumption}
t^\balpha (L, g)  \ge t^\balpha(L, h) - \eta
\end{equation}
for each $L \in \{F_1, \dots, F_k, J_1, \dots, J_\ell\}$ and vertex weight function $\balpha$ on $L$, then
\begin{equation}
	\label{eq:island-pf-goal}
t^\balpha(F, g) \ge t^\balpha(F, h) - (2\ell + k)\epsilon,
\end{equation}
for every vertex weight function $\balpha$ on $F$.

Write
\[
f_{\le t} (x) = \begin{cases}
	f(x) & \text{if } f(x) \le t, \\
	0 & \text{otherwise}
\end{cases}
\quad \text{and} \quad
f_{> t} (x) = \begin{cases}
	f(x) & \text{if } f(x) > t, \\
	0 & \text{otherwise.}
\end{cases}
\]

For each connector $(J, I) = (J_j,I_j)$ (temporarily dropping the subscript $j$ to avoid notational clutter),
writing
\[
g^\balpha_{J,I,>\delta^{-1}} = (g^\balpha_{J,I})_{>\delta^{-1}},
\]
we have, using $t(J \vee_I J, g) \le 1$ from \cref{eq:bounded}, that
\[
\int g^\balpha_{J,I,>\delta^{-1}}(x_I) \,d^\balpha x_I
\le \delta\, \int g_{J, I}^2 (x_I) \,d^\balpha x_I  
\leq \delta t(J \vee_I J, g) \le \delta.
\]
Thus, using \cref{eq:recur-assumption},
\begin{equation} \label{eq:connector-bounded-count}
\int g^\balpha_{J,I,\le \delta^{-1}}(x_I) \,d^\balpha x_I
\ge 
\paren{\int g^\balpha_{J,I} (x_I) \,d^\balpha x_I} - \delta
\ge 
\paren{\int h^\balpha_{J,I} (x_I) \,d^\balpha x_I} - \eta - \delta.
\end{equation}

\medskip
\emph{Step I. Swapping out the islands one at a time.}
\medskip

Write $F' = \cup_i F_i$ (islands without connectors).
We have
\begin{align*}
t^\balpha(F, g) 
&= \int g_F(x_F) \, d^\balpha x_F
\\
&= \int \prod_{i=1}^k g_{F_i}(x_{F_i}) \prod_{j=1}^\ell g_{J_j, I_j}^\balpha(x_{I_j}) \, d^\balpha x_{F'}
\\
&\ge \int \prod_{i=1}^k g_{F_i}(x_{F_i}) \prod_{j=1}^\ell g_{J_j, I_j, \le \eta_j^{-1}}^\balpha(x_{I_j}) \, d^\balpha x_{F'}
\\
&=
\int \paren{\int g_{F_k}(x_{F_k}) \prod_{j=1}^\ell g_{J_j, I_j, \le \eta_j^{-1}}^\balpha(x_{I_j}) d^\balpha x_{F_k}} \prod_{i=1}^{k-1} \paren{g_{F_i}(x_{F_i}) d^\balpha x_{F_i}}.
\end{align*}
Now, using \cref{eq:recur-assumption} for $F_k$ and noting that the inner integral inside the parenthesis has the form
$\int g_{F_k}(x_{F_k}) d^{\balpha'} x_{F_k} \cdot \prod_{j=1}^\ell \eta_j^{-1}$ for some other vertex weight function $\balpha'$ (absorbing the connector factors by using the fact that each connector uses at most one vertex from the island $F_k$), we have, continuing from above, that the last expression is
\[
\ge 
\int \paren{\int h_{F_k}(x_{F_k}) \prod_{j=1}^\ell g_{J_j, I_j, \le \eta_j^{-1}}^\balpha(x_{I_j}) d^\balpha x_{F_k} - \eta \prod_{j=1}^\ell \eta_j^{-1}} 
\prod_{i=1}^{k-1}  \paren{g_{F_i}(x_{F_i}) d^\balpha x_{F_i}}.
\]
Since $\eta \prod_{j=1}^\ell \eta_j^{-1} \le \epsilon$ by \cref{eq:eta} and $\int g_{F_i}(x_{F_i}) d^\balpha x_{F_i} \leq t(F_i, g) \le 1$ for each $i \in[k-1]$ by \cref{eq:bounded}, we can continue the above as
\[
\ge 
\int h_{F_k}(x_{F_k}) \prod_{i=1}^{k-1} g_{F_i}(x_{F_i}) \prod_{j=1}^\ell g_{J_j, I_j, \le \eta_j^{-1}}^\balpha(x_{I_j}) \, d^\balpha x_{F'} - \epsilon.
\]
We can now repeat this process to successively replace each remaining $g_{F_i}$ factor by $h_{F_i}$, losing at most an additive error of $\epsilon$ at each step.
(Note that even though we do not assume that $t(F_k, g) \le 1$, it is no longer needed, since what matters from now on is that $t(F_k, h) \le 1$ and this is automatically true for $h$, which takes values in $[0,1]$).
We may therefore continue the above as
\[
\ge \int \prod_{i=1}^k h_{F_i}(x_{F_i}) \prod_{j=1}^\ell g_{J_j, I_j, \le \eta_j^{-1}}^\balpha(x_{I_j}) \, d^\balpha x_{F'} - k\epsilon.
\]

\medskip
\emph{Step II. Swapping out the connectors one at a time.}
\medskip

Continuing, we have, applying \cref{eq:connector-bounded-count} to replace $g_{J_\ell, I_\ell, \le \eta_\ell^{-1}}^\balpha(x_{I_\ell})$ by $h_{J_\ell, I_\ell}^\balpha(x_{J_\ell})$ (here we are applying \cref{eq:connector-bounded-count} for each fixed $x_{F \setminus J_\ell}$ and with a different $\balpha$ which absorbs additional factors;
this step works only because each $J_\ell$ intersects each of $F_1, \dots, F_k$, $J_1, \dots, J_{\ell-1}$ in at most one vertex and all these intersections are contained in $I_\ell$), that the last expression above is
\[
\ge \int \prod_{i=1}^k h_{F_i}(x_{F_i}) \cdot 
h_{J_\ell, I_\ell}^\balpha(x_{I_\ell})
\prod_{j=1}^{\ell-1} g_{J_j, I_j, \le \eta_j^{-1}}^\balpha(x_{I_j}) \, d^\balpha x_{F'}
- (\eta + \eta_\ell)\prod_{j=1}^{\ell-1}\eta_j^{-1} - k\epsilon.
\]
We have $(\eta + \eta_\ell)\prod_{j=1}^{\ell-1}\eta_j^{-1} \le 2\epsilon$ by \cref{eq:eta}. 
Continuing, we can replace 
$g_{J_j, I_j, \le \eta_j^{-1}}^\balpha(x_{I_j})$ by $h_{J_j,I_j}^\balpha(x_{I_j})$ one at a time in \emph{decreasing order of $j$}, so that the additive error at $j$ is at most $(\eta + \eta_j)\eta_1^{-1} \cdots \eta_{j-1}^{-1} \le 2\epsilon$ (this is why we need $\eta_1, \dots, \eta_\ell$ to be rapidly decreasing).
Finally, we can continue the above as
\begin{align*}
&\ge \int \prod_{i=1}^k h_{F_i}(x_{F_i}) \prod_{j=1}^\ell h_{J_j, I_j}^\balpha(x_{I_j}) \, d^\balpha x_{F'} - (k+2\ell)\epsilon
\\ &= t^\balpha(F, h) - (k+2\ell)\epsilon,
\end{align*}
thereby proving \cref{eq:island-pf-goal}.
\end{proof}

\section{Concluding remarks}

We conclude by exploring some of the problems that arose from our study of countability.

\subsection*{Classifying countable graphs.} 
We have made partial progress on our \cref{que:main} by producing a family of graphs $F$ for which there is an $F$-counting lemma in $C_4$-free graphs. However, our results are likely far from a complete classification. We saw one necessary condition on any such $F$ in \cref{rem:nec}, namely, that $F$ should have girth at least $5$. It also seems necessary that the $2$-density of $F$ should be less than $2$, that is, that any subgraph $F’$ of $F$ should satisfy $|E(F’)| \leq 2|V(F’)| - 4$. In particular, this would imply that any $d$-regular countable graph has $d \leq 3$.

Though not a formal proof, the intuition here is that the number of copies of $F'$ in our $C_4$-free graph should not be smaller than the number of edges (otherwise, we can delete all copies of $F'$, and hence $F$, by removing an edge from each copy) and, for a random graph of the same density $n^{-1/2}$, the condition that the $2$-density be less than $2$ is necessary for this to hold. Most likely, the true conditions for countability are even more stringent than this argument suggests. Perhaps resolving the cases highlighted in \cref{open:dod-pet} would be a good starting point for further progress. 

We remark in passing that we expect any progress on \cref{que:main} to also impinge on the closely related question where we assume that there are $o(n^2)$ copies of $C_4$ in our $n$-vertex graph rather than none. Indeed, the arguments in~\cite{CFSZ} showing that $C_5$ is countable apply in this more general situation and the proofs here may also be adapted to this context. We suspect that the same will be true of any countable graph.

\subsection*{Variations on countability.}
There are several variants of our basic question which may be interesting. For instance, for which graphs $F$ is there a two-sided counting lemma in $C_4$-free graphs? Our results are fundamentally one-sided, so new ideas are probably necessary to make progress on this question. However, we do know that for $F$ to satisfy a two-sided counting lemma, it must, at the very least, be tame. As observed in \cref{ex:subdiv}, this already rules out two-sided counting for the family of subdivisions $K’_t$ with $t \geq 5$. 

Another natural variant is to ask which graphs $F$ have an $F$-counting lemma in $H$-free graphs when $H$ is a bipartite graph other than $C_4$? Our arguments apply just as well to $K_{2,t}$-free graphs as they do to $C_4$-free graphs, but further extensions are less obvious. We do expect our methods to extend to prove counting lemmas in $C_{2k}$-free graphs for any $k \geq 3$, but here the real difficulty passes back to the regularity side. Indeed, in order to apply a $C_{2k+1}$-counting lemma in $C_{2k}$-free graphs to prove a  corresponding removal lemma, we also need to show that any regular partition of a $C_{2k}$-free graph has few edges between irregular pairs. However, we do not at present know how to do this for any $k \geq 3$. As in~\cite{CFSZ}, resolving this issue would have several consequences. To give just one example, it would allow us to show that any $3$-uniform hypergraph with $n$ vertices and girth greater than $2k+1$ has $o(n^{1+1/k})$ edges, extending both the classic Ruza--Szemer\'edi theorem~\cite{RS78}, which is equivalent to the case $k = 1$, and a recent result of the authors~\cite[Corollary 1.10]{CFSZ} resolving the case $k = 2$.

\section*{Acknowledgments}

Part of this work was completed in the summer of 2019 while Yufei Zhao was generously hosted by FIM (the Institute for Mathematical Research) during a visit to Benny Sudakov at ETH Z\"urich.

% \bibliographystyle{amsplain0.bst}
% \bibliography{c4free-counting.bib}

\end{document}